\documentclass[10pt]{amsart}
\usepackage{amsfonts}
\usepackage{amssymb}
\usepackage{amsmath}
\usepackage[arrow,matrix,curve]{xy}

\usepackage[normalem]{ulem}
\usepackage{color}

\newcommand{\C}{{\mathbb C}}
\newcommand{\R}{{\mathbb R}}

\newtheorem{theorem}{Theorem}[section]
\newtheorem{lemma}[theorem]{Lemma}
\newtheorem{remark}[theorem]{Remark}
\newtheorem{example}[theorem]{Example}
\newtheorem{corollary}[theorem]{Corollary}
\newtheorem{proposition}[theorem]{Proposition}
\newtheorem{definition}[theorem]{Definition}

\def\cal{\mathcal}

\newcommand{\calh}[0]{{\cal H}}

\newcommand{\call}[0]{{\cal L}}
\newcommand{\calg}[0]{{\cal G}}

\newcommand{\cala}[0]{{\cal A}}
\newcommand{\cale}[0]{{\cal E}}
\newcommand{\calk}[0]{{\cal K}}

\newcommand{\cali}[0]{{\cal I}}
\newcommand{\cals}[0]{{\cal S}}
\newcommand{\calm}[0]{{\cal M}}

\newcommand{\calc}[0]{{\cal C}}

\newcommand{\calt}[0]{{\cal T}}

\begin{document}

\title[A certain Baum--Connes map]{A note on a certain Baum--Connes map for inverse semigroups}
\author[B. Burgstaller]{Bernhard Burgstaller}
\email{bernhardburgstaller@yahoo.de}

\maketitle

\begin{abstract}
Let $G$ denote a countable inverse semigroup.
We construct a kind of a Baum--Connes map $K(\tilde A \rtimes G) \rightarrow K(A \rtimes G)$
by a categorial approach via localization of triangulated categories, developed by R. Meyer and
R. Nest for groups $G$.
We allow the coefficient algebras $A$ 
to be in a special class of algebras called fibered $G$-algebras.
This note continues and fixes our preprint ``Attempts to define a Baum--Connes map via localization
of 
categories for inverse semigroups''.

\end{abstract}

\tableofcontents

\section{Introduction}

In \cite{meyernest}, R. Meyer and R. Nest found an equivalent definition of the Baum--Connes map \cite{baumconneshigson1994}.
The new definition defines the Baum--Connes map for a coefficient $G$-algebra $A$ as a homomorphism
$$K(\tilde A \rtimes G) \rightarrow K(A \rtimes G)$$
of $K$-theory groups, where $\tilde A$ is a certain approximation for $A$.
To be precise, $\tilde A$ sits in the triangulated subcategory of $KK^G$ which is generated
by induced algebras of the form $\mbox{Ind}_H^G(B)$ for some compact subgroups $H$ of $G$.
This gives a potentially possible way to compute the left hand side $K(\tilde A \rtimes G)$
of the new Baum--Connes map by homological means.
If $\tilde A=\mbox{Ind}_H^G(B)$ then one may directly use
\begin{equation}   \label{comp1}
K(\tilde A \rtimes G) = K(\mbox{Ind}_H^G(B) \rtimes G) \cong K(B \rtimes H) \cong K^H(\C,B)
\end{equation}
by Green's imprimitivity theorem \cite{green1978} and the Green--Julg isomorphism.

In \cite{attempts} we tried to define a Baum--Connes map in an analogous way for inverse semigroups $G$,
but ran into serve problems which at the end turned out to rely on the fact
that we assumed a wrong right adjoint functor to the induction functor.


%

In this note we shall define a correct right adjoint functor, called the fibered restriction functor,
and 
is given by
$\mbox{R}_G^H(A)=\oplus_{e \in E_H} A_{\varepsilon_e}$.
It turns out, however, that it will only work on the subcategory of $KK^G$
which is generated by fibered $G$-algebras.
These are $G$-algebras of the form $\bigoplus_{e \in E} A_{\varepsilon_e}$.
($C_0(X)$ is not fibered.)
Hence we can build a Baum--Connes map only for fibered coefficient algebras.

%

The idea is as follows. 
We analyzed the $G$-action on $\mbox{Ind}_H^G(A)$ and 
interpreted it as a groupoid action.
Then we made the simple observation (encoded in Lemma \ref{lemma12}) that a simple characteristic function
$1_{g (1-e_1) \ldots (1-e_n)} \in \mbox{Ind}_H^G(A)$ 
for $g \in G, e_i \in E$
has carrier the single point $\varepsilon_{g g^*}$ in the base space of the groupoid.
(The character 
is defined by $\varepsilon_e(f) = 1_{\{f \ge e\}}$.)   
%
The $G$-action on $\mbox{Ind}_H^G(A)$ then just shifts $1_g$ as usual, that is,
$h(1_g)$ has carrier the single point $\varepsilon_{h g g^* h^*}$ for $h \in G$.
In other words, $\mbox{Ind}_H^G(A)$ is a fibered $G$-algebra.
We shall not lay out all the heuristical idea, 
but it is encoded in this paper and see Remark \ref{remark2} and Example \ref{example}.




The main work of this note is the definition of the fibered restriction functor in \ref{deffiberedrestriction} and the proof that it is
right adjoint to the induction functor in \ref{propAdjoint}.
We shall also introduce a new $\ell^2(G)$-space as a technical tool in \ref{definition1}.
We can use blueprints from \cite{attempts} to obtain the fact that $KK^G$ is triangulated for fibered $G$-algebras.
(This is of course analogous to the proof by \cite{meyer} and \cite{meyernest}.)
Fortunately, Ralf Meyer pointed out to us his work \cite{meyerTriangulatedII}
soon after publishing \cite{attempts}, which then immediately yields the approximation
$\tilde A$ mentioned above and thus the Baum--Connes map, and we need not go through the technicalities
to define such an approximation as in \cite{attempts}.

The resulting Baum--Connes map given in Definition \ref{defBaumConnesmap} is justified in so far as computation (\ref{comp1})
works also for inverse semigroups and 
Sieben's crossed product \cite{sieben1997}
by \ref{thm1}.   
It
has, however, usually 
less potential for computing the full crossed product $\C \rtimes G$,
see Remark \ref{remark3}.
%

Sections \ref{sectionSomelemmas}-\ref{sectionAdjoint} essentially occupy the definition of the
fibered restriction functor and the verification that it is right adjoint to the induction functor.
The last Sections \ref{sectionAdaption}-\ref{sectionBC} cover 
slight adaption and collection of known results.



\section{Some notation}

In this note, $G$ denotes a countable discrete inverse semigroup.
We define $G$-actions on $C^*$-algebras and $G$-equivariant $KK$-theory $KK^G$ as in \cite{burgiSemimultiKK}.
In case of an inverse semigroup $G$ the formal definitions simplify slightly, for which we refer to
\cite{attempts}.
The letter $C^*_G$ will denote the category of $G$-algebras with their $G$-equivariant $*$-homomorphisms. 
$KK^G$ will stand for the Kasparov category consisting of $G$-algebras as object class and
Kasparov groups as morphism sets.   
For convenience of the reader we recall the notion of a $G$-action on a $C^*$-algebra. 

\begin{definition}
{\rm
A {\em $G$-algebra} is a $C^*$-algebra $A$ equipped with a semigroup homomorphism
$\alpha:G \rightarrow  \mbox{End}(A)$ such that $\alpha_{e} (a) b = a \alpha_{e} (b)$
for all $a,b \in A$ and $e \in E$.
}
\end{definition}

Throughout we write $g(a) := \alpha_g(a)$.
The letter $E$ stands for the set of idempotent elements of $G$,
and $C^*(E)$ for the abelian $C^*$-algebra it 
generates.    
We write $C^*(E) \cong C_0(X)$ by Gelfand's theorem, where $X$ denotes the character space of $C^*(E)$.
Every $e \in E$ defines a
character $\varepsilon_e \in X$ by 
$\varepsilon_e(f) = 1_{\{f \ge e\}}$ for all $f \in E$.
The character set $\{\varepsilon_e\}_{e \in E}$ is dense in $X$.
See \cite{paterson} or \cite[2.6]{khoshkamskandalis2004} for more details.

The algebra $C^*(E)$ is a $G$-algebra under the $G$-action $g(e)= g e g^*$.
For heuristical comments we shall consider a bigger function space than $C_0(X)$.
Let $\C^X$ be the set of functions from $X$ to $\C$. It is endowed with the $G$-action
induced by the maps $g:X \rightarrow X$ given by $(g(x)) (e) = x(g^* e g)$ for all $x \in X, g \in G,  e\in E$.
This is consistent with the $G$-action on $C_0(X)$ defined before, that is, $C_0(X) \subseteq \C^X$ $G$-equivariantly.
We shall write $1_x$ for the characteristic function $1_{\{x\}}$ of a single point $x$.

Every $G$-algebra $A$ is equipped with a $G$-equivariant $*$-homomorphism $C_0(X) \rightarrow Z\calm(A)$
(center of the multiplier algebra of $A$).
We denote the $C_0(X)$-balanced tensor product of $G$-algebras or $C_0(X)$-algebras $A,B$ by $A \otimes^{C_0(X)} B
:= (A \otimes B)/I$, where $I$ is the ideal generated by $e(a) \otimes b - a \otimes e(b)$ for all $a,b \in A,e \in E$.
We write $A_x = \C\{1_x\} \otimes^{C_0(X)} A$ for the fiber of $A$ in $x \in X$.

The universal (or occasionally an unspecified) crossed product \cite{khoshkamskandalis2004} is denoted by $A \rtimes G$, and Sieben's (``compatible'' universal) crossed
product \cite{sieben1997} by $A \widehat \rtimes G$.
We sometimes consider another inverse semigroups
\begin{eqnarray*}
\tilde E &=& \{e_0 (1-e_1) \cdots (1-e_n) \in \C \rtimes G |\, e_i \in E, n \ge 1\},  \\
\tilde G &=& \{ g p \in \C \rtimes G| \, g \in G, p \in \tilde E\}
\end{eqnarray*}
as subinverse semigroups of $\C \rtimes G$ under multiplication and involution.

For an assertion $\cala$ we write $[\cala]$ for the real number which is $1$ if $\cala$ is true,
and $0$ 
otherwise.
In case that we have given another inverse semigroup $H$ we shall specify the associated sets $E$ and $X$
by writing $E_H$ and $X_H$.
Usually $H$ denotes a finite subinverse semigroup of $G$ and
$\mbox{Res}_G^H$ 
the usual restriction functor.
For possibly further needed details we refer to \cite{attempts}.

\section{Some lemmas}

\label{sectionSomelemmas}

In this section we observe some 
central lemmas.

\begin{lemma}   \label{lemma12}
Let 
$f \in E$ and $p:=e_0 (1-e_1) \ldots (1-e_n)$ a nonzero element in $\tilde E$
($e_i \in E$). 
Then
$$f \ge p 
\quad \Leftrightarrow \quad
f \ge e_0.$$

\end{lemma}

\begin{proof}
$\Leftarrow$ is clear.
$\Rightarrow$: 
By expanding the brackets in $f p = p$ we get
%
\begin{equation}  \label{q1}
f e_0 - f e_0 e_1 + f e_0 e_1 e_2 \pm \ldots = e_0 - e_0 e_1 + e_0 e_1 e_2 \pm \ldots
\end{equation}
Since $e_0 (1-e_1) \ldots (1-e_n) \neq 0$, $e_0 e_i \neq e_0$ for all $i \ge 1$.
Hence
$$e_0 e_1, e_0 e_1 e_2 , e_0 e_2, \ldots < e_0.$$
Since the projections $E$ are linearly independent in $C^*(E)$, the only possibility that
(\ref{q1}) is true is that $f e_0 = e_0$. That is, $f \ge e_0$.
\end{proof}

\begin{definition}
{\rm
For a nonzero $p=e_0 (1-e_1) \ldots (1-e_n) \in \tilde E$ ($e_i \in E$) set 
$$\sigma(p) := e_0,$$
the {\em leading coefficient} of $p$.
This yields a map $\sigma:\tilde E \backslash \{0\} \rightarrow E$.

%
}
\end{definition}

Note that the last definition is well-defined since $e_0$ is the unique minimal projection in $E$ such that
$e_0 \ge p$ by Lemma \ref{lemma12}.
The next lemma shows how the leading coefficient already uniquely determines certain sets of projections
in $\tilde E$.


\begin{lemma}  \label{lemma11}

Let $H 
 \subseteq G$ be a finite subinverse semigroup.
View $X_H \subseteq \tilde E_H$ as the set of all minimal nonzero projections of $\tilde E_H$.
Then the map $\sigma|_{X_H}: X_H \rightarrow E_H$ is a bijection.

\end{lemma}


\begin{proof}
(a)
%
Let $p=e_0(1-e_1) \ldots (1-e_n)$ and $q=f_0(1-f_1) \ldots (1-f_m)$ be nonzero elements in $X_H$
for $e_i,f_j \in E_H$.
Assume that $\sigma(p) = e_0 = f_0 = \sigma(q)$.
Let $i \ge 1$. Note that $e_i \not \ge e_0 = f_0$. Thus $e_i f_0 < f_0$.
Hence, necessarily $(1-e_i) q = q$. Consequently $p \ge q$. Similarly $p \le q$.
%
This shows injectivity of $\sigma$.
%
To prove surjectivity, consider
$e \in E_H$. Set $p= e \prod_{f \in E_H, f \not \ge e}(1-f)$. Then $p \in X_H$ and $\sigma(p)=e$.
\end{proof}

The restriction $\sigma|_{X_H}$ will also be denoted by $\sigma_H$.
Since $H_H = H$ as sets (see Definition \ref{defGH} below), we shall also write $H/H$ for $H_H/H$.

\begin{lemma}    \label{lemma13}
\begin{itemize}

\item[(a)]
One has $\sigma(p) \ge p$ for all $p \in \tilde E$.

\item[(b)]
We have $\sigma( g p g^*) = g \sigma(p) g^*$ for all $g \in G$ and $p \in \tilde E$. 

\item[(c)]
Furthermore, $\sigma^{-1}_H(e) \le e$ and $\sigma^{-1}_H(h e h^*) = h \sigma^{-1}_H(e) h^*$ for all $h \in H$ and $e \in E_H$.

\item[(d)]
Moreover, $\sum_{h \in H/H} \sigma^{-1}_H(h h^*) = 1_H$.

\item[(e)]
The map $\sigma$ is multiplicative.

\end{itemize}
\end{lemma}

\begin{proof}
(a) and (e) are clear.
(b) Since $g p g^* = g e_0 g^* (1- g e_1 g^*) \ldots (1-g e_n g^*)$.
(c) follow from (b).
(d) Note that every element of $E_H$ is of the form $h h^*$ for some $h \in H$, and $h_1 \equiv h_2$
for $h_i \in H$ if and only if $h_1 h_1^* = h_2 h_2^*$ (if and only if $h_1 = h_2 h_2^* h_1$).
Hence $\sum_{h \in H/H} \sigma^{-1}_H(h h^*) = \sum_{p \in X_H} p = 1_H$ by Lemma \ref{lemma11}.
\end{proof}

\section{The induction functor}

In this section we recall the definition of a $G$-algebra which is induced by an $H$-algebra
for a finite subinverse semigroup $H$ of $G$. We shall use a formally slightly modified but equivalent definition
than in \cite{burgiGreen},
see Corollary \ref{corollary1}. 
The reason is the observation made in Lemma \ref{lemma14} that we may 
change slightly a set called $G_H$.

%


\begin{definition}		\label{defGH}
{\rm
Let $H \subseteq G$ be a finite subinverse semigroup.
We define $G_H \subseteq G$ as
$$G_H:= \{g \in G|\, g^*g \in H\}.$$
We endow $G_H$ with an equivalence relation: $g \equiv l$ ($g,l \in G_H$) if and only if $g h = l$ for
some $h \in H$ with $g^* g = h h^*$.
The set-theoretical quotient $G_H/\equiv$ is denoted by $G_H/H$.
}
\end{definition}

\begin{definition}
{\rm
Let $c_0(G)$ be the usual commutative $C^*$-algebra of complex-valued functions on $G$ vanishing at infinity
($G$ being
discrete), endowed with the $G$-action $(k(f))(g) = f(k^* g) \, [k k^* \ge g g^*]$ for all $f \in c_0(G)$
and $k,g \in G$. This turns $c_0(G)$ to a $G$-algebra.
}
\end{definition}

We denote by $c_0(G_H)$ the $G$-subalgebra of $c_0(G)$ consisting of all functions vanishing outside of $G_H$.
%
We similarly define $c_0(G_H/H)$ as the usual commutative $C^*$-algebra, 
endowed with the $G$-action
$(k(f))(g H) = f(k^* g H) \, [k k^* \ge g g^*]$ for all $f \in c_0(G_H/H)$
and $k,g \in G$,
which turns it to a $G$-algebra.


\begin{definition}   \label{defInductionAlg}
{\rm
Let $H \subseteq G$ be a finite subinverse semigroup
and $D$ a $H$-algebra.
Define, similar as in \cite[\S 5 Def. 2]{kasparov1995},
\begin{eqnarray*}
\mbox{Ind}_{H}^G(D)  &:=&  \{ f: G_H \rightarrow D \, | \,\, \forall g \in G_H : \,
\forall h \in H \mbox{ with } g^*g = h h^*:   \\
&& \quad  f(g h) = \sigma^{-1}_H(h^* h) \,h^{*}(f(g)),
\quad \|f(g)\| \rightarrow 0 \mbox{ for } g H \rightarrow \infty \,\} .  
\end{eqnarray*}
It is a $C^*$-algebra under the pointwise operations and the supremum's norm
$\|f\|=\sup_{g \in G_H} \|f(g)\|$, and becomes a $G$-algebra
under the $G$-action
$(k( f))(h) := 
f ( k^{*} g) \, [k k^* \ge g g^*]$
for all $k \in G, g \in G_H$ and $f \in \mbox{Ind}_H^G(D)$.
}
\end{definition}


\begin{definition}
{\rm
Let $H \subseteq G$ be a finite subinverse semigroup.
By the universal property
of $KK^H$-theory \cite{burgiUniversalKK},
there exists an {\em induction functor} $\mbox{Ind}_H^G: KK^H \rightarrow KK^G$ induced by the functor
$F:C^*_H \rightarrow C^*_G$ given by $F(A) = \mbox{Ind}_H^G(A)$ for $H$-algebras $A$.
For more details see \cite{attempts}.
}
\end{definition}


A key 
motivation for the definition of an induction algebra is the
following variant of Green's imprimitivity theorem \cite{green1978}.

\begin{theorem}[\cite{burgiGreen}]    \label{thm1}
Let $H \subseteq G$ be a finite subinverse semigroup and $A$ a $H$-algebra.
Then the $C^*$-algebras $\mbox{Ind}_H^G(A) \widehat \rtimes G$ and $A \widehat \rtimes H$
are Morita equivalent.
\end{theorem}

In \cite{burgiGreen} we defined $G_H$ and 
$\mbox{Ind}_H^G(D)$ slightly differently.
But as already mentioned, both definitions are equivalent.
To explain this in detail,
re-denote $G_H$ of \cite{burgiGreen} as $G_H'$. Define $\calg$ as the finite groupoid associated to $H$.
That is,
one sets $\calg^{(0)} := 
X_H \subseteq \tilde E_H \subseteq \tilde G$ 
and
$\calg= \{h p \in \tilde G|\,h \in H, p \in \calg^{(0)}\} \backslash \{0\} \subseteq \tilde G$.

One then sets
$G_H' 
= \{t \in \tilde G\,|\, t^* t \in \calg^{(0)}\}$. 
An equivalence relation on $G_H'$ is given by $t \equiv s$ ($t,s \in G_H'$) if and only if
there exists an $h \in \calg$ (equivalently: $h \in H$) such that $t h = s$.

\begin{lemma}   \label{lemma14}
The map $\delta: G_H'  \rightarrow G_H$ given by $\delta (g p) = g \sigma(p)$ ($g \in G, p \in \calg^{(0)}$)
is a bijection which respects the equivalence relations 
in both directions.

The inverse map is given by
$\delta^{-1}(g) = g \sigma^{-1}_H(g^* g)$ for all $g \in G_H$.
\end{lemma}

\begin{proof}
Given $g p \in G_H'$ we observe that
\begin{equation}    \label{equ22}
g^* g \ge \sigma(p) \ge  p .
\end{equation}
Indeed, notice that $g^* g \ge p$ because by definition $(g p)^*(g p)$ is in $X_H$ and so can only be $p$.
Thus by Lemma \ref{lemma12}, $g^* g \ge \sigma(p)$.
It is then straightforward to check with Lemma \ref{lemma12} that $\delta$ and $\delta^{-1}$ are inverses to each other.


We are going to discuss the equivalence relations.
Suppose that $g p \equiv g' p'$ in $G_H'$. Then there exists a $h \in H$
such that $g p h = g' p'$, or $g h (h^* p h) = g' p'$,
and $h h^* \ge p$ 
(recalling (\ref{equ22})).
Applying $\delta$, we get $g h \sigma(h^* p h) = g  \sigma(p) h = g' \sigma(p)$
with Lemma \ref{lemma13}.
In other words, $\delta(g p) h'  = \delta (g' p')$ for $h':=\sigma(p) h \in H$.
By Lemma \ref{lemma12}, $h h^* \ge \sigma(p)$, and thus $h' {h'}^* = \delta(gp) \delta(gp)^*$.
Hence $\delta(g p) \equiv \delta (g' p')$ in $G_H$.


If $g \equiv g'$ in $G_H$ then there exists a $h \in H$ with $g h= g'$ and $g^*g = h h^*$.
Recall that $\delta^{-1}(g ) = gp$ with $\sigma(p)=g^* g$.
Hence $\delta^{-1}(gh) = \delta^{-1}(g') = gh p = g' p'$ with $h^* h = \sigma(p)$ and $g'^* g' = \sigma(p')$.
Notice that $g^*g = h h^* = \sigma( h p h^*)$
with Lemma \ref{lemma13}. Consequently $\delta^{-1}(g) h = g h p h^* h  = \delta^{-1}(g')$.
Hence $\delta^{-1}(g) \equiv \delta^{-1}(g')$ in $G_H'$.
%
\end{proof}

Let us re-denote the induction algebra $\mbox{Ind}_H^G(D)$ of \cite{burgiGreen} 
as $\mbox{Ind}_H^G(D)'$.

\begin{corollary}   \label{corollary1}
There is a $G$-equivariant isomorphism $\mbox{Ind}_H^G(A) \rightarrow \mbox{Ind}_H^G(A)'$.
%
\end{corollary}

\begin{proof}

The isomorphism $\varphi$ is given by
$\varphi(f)(t) = 
f(\delta(t))$
for all $f \in \mbox{Ind}_H^G(A)$ and $t \in G_H'$.
To see that it is well-defined,
consider $t = gp$ and $h = k p \in \calg$ for $g \in G, p \in X_H, k \in H$.
Note that $g^*g, k k^* \ge p$ and thus $g^* g , k k^* \ge \sigma(p)$ by Lemma \ref{lemma12}.
Then
$(\varphi(f))(t h) = f( \delta(th)) = f( \delta(g k k^* p k)) = f(g k \sigma(k^* p k)) 
= f ( g \sigma(p) k) = \sigma_H^{-1}(k^* \sigma(p) k) k^*  \sigma(p) f( g \sigma(p))
= k^* p f(\delta(t)) = h^*(\varphi(f)(t))$
by Lemma \ref{lemma13}.

Surjectivity may be observed by setting $h=p$ and $k=1$ in the last computation.
For the $G$-invariance just notice that an $f \in E$ satisfies $f \ge g p g^*$ if and only if $f \ge \sigma(g p g^*)= g \sigma(p) g^*$
by Lemma \ref{lemma12}.
\end{proof}

\section{The $\ell^2(G)$-space}

In this section we will shall define 
fibered $G$-algebras and 
an $\ell^2(G)$-space as a tool for working with such algebras.

\begin{definition}		
{\rm
Let $\varepsilon(E)$ denote the commutative $C^*$-algebra $c_0(E)$
($E$ being discrete).
We turn $\varepsilon(E)$ to a $G$-algebra by setting
$(g(f))(e) = f(g^* e g)\,[g g^* \ge e]$ for all $g \in G$, $f \in \varepsilon(E)$ and $e \in E$.
}
\end{definition}

Equivalently we may define the $G$-action by $g(1_e) = 1_{g e g^*} [g g^* \ge e]$ for all $e \in E$ ($1_e := 1_{\{e\}} \in \varepsilon(E)$).
The algebra $\varepsilon(E)$ will be used as a replacement for $\C$ as utilized in group equivariant $C^*$-theory, see Lemma \ref{lemma1}.

\begin{remark}    \label{remark2}
{\rm
Heuristically and even exactly if we like, 
we view the characteristic function $1_e\in \varepsilon(E)$ as the characteristic function $1_{\varepsilon_e} \in \C^X$
of the point $\varepsilon_e \in X$.
In other words, $\varepsilon(E)$ is $G$-equivariantly $*$-isomorphic to the
$G$-invariant $G$-subalgebra of
$\C^X$ generated by the simple functions $1_{\varepsilon_e}$ via the map $1_e \mapsto 1_{\varepsilon_e}$.

Fibers of a $C_0(X)$-algebra $A$ may be computed by
$$A_{\varepsilon_e} = \C\{1_e\} \otimes^{C_0(X)} A.$$

Consequently
\begin{equation}    \label{dsum}
\bigoplus_{e \in E} A_{\varepsilon_e} = \varepsilon(E) \otimes^{C_0(X)} A.
\end{equation}
}
\end{remark}

\begin{lemma}     \label{lemma2}
$\varepsilon(E)$ is a $G$-algebra.
\end{lemma}

\begin{proof}
We have $(gh(f))(e)= f(h^* g^* e g h)\,[g h h^* g^* \ge e]$ and $(g (h(f)))(e) = f(h^* g^* e g h)\, [g g^* \ge e]
\, [h h^* \ge g^* e g]$ for $g,h \in G, e \in E$ and $f \in \varepsilon(E)$.
Now $g h h^* g^* \ge e$ implies $h h^* \ge g^* g h h^* g^* g \ge g^* e g$ and $g g^* \ge g h h^* g^* \ge e$
and similarly reversely. Hence $(gh)(f)= g (h (f))$.
Further, $(e_1(f_1) f_2)(e_2)= f_1(e_1 e_2) f_2(e_2) \,[e_1 \ge e_2] = (f_1 \cdot e_1(f_2))(e_2)$
for $e_i \in E$.
\end{proof}


\begin{lemma}   \label{lemma3}
Let $h \in G$.
There are bijections
\begin{eqnarray*}
&& \{ g \in G|\, g g^* = h^* h \} \rightarrow \{ g \in G| \, g g^* = h h^* \} : g \mapsto h g , \\
&& \{ g \in G|\, g g^* \le h^* h \} \rightarrow \{ g \in G| \, g g^* \le h h^* \} : g \mapsto h g .
\end{eqnarray*}
Similar things can be said for the right multiplication $g \mapsto g h$.
\end{lemma}


\begin{definition}		\label{definition1}   
{\rm


Let $c_c(G)$ denote the linear space consisting of functions $G \rightarrow \C$ with finite support.
We turn $c_c(G)$ to a right $\varepsilon(E)$-module by setting 
$$(\xi  f)(g)= \xi(g) f(g g^*)$$ 
for all $\xi \in c_c(G), f \in c_0(E)$ and $g \in G$.
This module is endowed with an 
$\varepsilon(E)$-valued inner product 
given by
$$\langle \xi, \eta \rangle (e)= \sum_{g \in G, g g^* =e} \overline{\xi(g)} \eta(g) . $$  


The space $c_c(G)$ will be equipped with the $G$-action
$$(h \xi)(g)= \xi (h^* g) \, [ h h^* \ge g g^*]$$
for all $\xi \in c_c(G)$ and $h,g \in G$.

The closure of $c_c(G)$ under the norm induced by the inner product is a $G$-Hilbert 
$\varepsilon(E)$-module
denoted by $\ell^2(G)$.
}
\end{definition}

\begin{lemma}
The space $\ell^2(G)$ is a $G$-Hilbert $\varepsilon(E)$-module.

\end{lemma}

\begin{proof}
It is obvious that the inner product is positive definite.
The module structure is straightforward to check, and for the 
admissibility of the $G$-action confer the proof of Lemma \ref{lemma2}.
It is well-known that $\ell^2(G)$ is consequently a Hilbert $\varepsilon(E)$-module.
Also the $G$-action extends to $\ell^2(G)$ by continuity of linear operators:
$$\|h(\xi)\|^2_{c_c(G)}
= \sup_{e \in E} \|  \sum_{g \in G, g g^* =e} \overline{\xi(h^* g)} \eta(h^* g) \,[ h h^* \ge e]\|
\le \|\xi\|^2_{c_c(G)}$$
by Lemma \ref{lemma3}.
\end{proof}


Similarly to $\ell^2(G)$ we may define a $G$-Hilbert $\varepsilon(E)$-module $\ell^2(G_H/H)$.
(It may be regarded as the submodule of $\ell^2(G)$ consisting of all functions vanishing outside of $G_H$ and
being constant on equivalence classes.)

\begin{lemma}   \label{lemma1}
There
are $G$-equivariant $*$-isomorphisms
$$A \otimes^{C_0(X)} \varepsilon(E) \cong  A $$
for the $G$-algebras $A=\varepsilon(E), c_0(G),c_0(G_H),c_0(G_H/H), \mbox{Ind}_H^G(D)$.
\end{lemma}

\begin{proof}
One checks that $a \otimes b \mapsto ab$ realizes these isomorphisms,
where $ab$ is the module multiplication used in Definition \ref{definition1}.
\end{proof}

%
%

\begin{definition}    \label{deffibered}
{\rm
A {\em fibered $G$-algebra} is a $G$-algebra of the form 
$\varepsilon(E) \otimes^{C_0(X)} A$
up to isomorphism
for 
some $G$-algebra $A$.
}
\end{definition}

If $G$ is finite then every $G$-algebra is fibered. 
Indeed, (\ref{dsum}) is $A$ by the fact that $\varepsilon_E = X$.

\section{The fibered restriction functor}



Let $H$ be a finite subinverse semigroup of $G$.
%
Regard $\varepsilon(E_H)$ as a $G$-subalgebra of $\varepsilon(E)$ in a canonical way.
We shall denote by $H \cdot E \subseteq G$ the subinverse semigroup of $G$ generated by $H$ and $E$.
Note that $\varepsilon(E_H)$ is a $H \cdot E$-invariant subalgebra of $\varepsilon(E)$.

We define the fibered restriction as $\mbox{R}_G^H(A)=\oplus_{e \in E_H} A_{\varepsilon_e}$
for $G$-algebras $A$.
In another way me say:

\begin{definition}   \label{deffiberedrestriction}
{\rm
Let $H \subseteq G$ be a finite subinverse semigroup.
Define 
the {\em fibered restriction} functor 
$R_G^H : KK^G \rightarrow KK^H$
by
\begin{eqnarray}
\mbox{R}_G^H (A) &=& 
\mbox{Res}_{G}^H \big (\varepsilon(E_H  
) \otimes^{C_0(X)}
A
 \big )    \label{ident8} \\ 
&:=&
\mbox{Res}_{H \cdot E}^H \big (\varepsilon(E_H  
) \otimes^{C_0(X)}
\mbox{Res}_G^{H\cdot E}( A)
 \big )  \label{ident6}  
\end{eqnarray}
for an object $A$ in $KK^G$.
The meaning of (\ref{ident8}) is more precisely repeated in (\ref{ident6}).
As usual, for a morphism
$[(\pi,\cale,T)] \in KK^G(A,B)$ one sets
$$\mbox{R}_G^H [(\pi,\cale,T)] = \big [\big (1 \otimes \pi,
\mbox{Res}_G^H \big(\varepsilon(E_H) \otimes^{C_0(X)}  \cale \big ),1 \otimes T \big ) \big ] \in
KK^H \big (\mbox{R}_G^H (A),\mbox{R}_G^H (B) \big ).$$
}
\end{definition}


\begin{lemma}    \label{lemma8}
Let $H \subseteq G$ be a finite subinverse semigroup and $B$ a $G$-algebra.
There is a $G$-equivariant isomorphism
$$\mu:Ind_H^G R_G^H (B) \rightarrow  c_0(G_H/H) \otimes^{C_0(X)}
B$$
given by
$$\mu(f) = \sum_{g \in G_H/H} 1_{g H} \otimes g (f(g)) \quad \in \quad c_0(G_H/H) \otimes^{C_0(X)} \varepsilon(E) \otimes^{C_0(X)} B, $$
%
where $f \in Ind_H^G R_G^H (B)$ is interpreted as a function $f:G_H \rightarrow 
\varepsilon(E) \otimes^{C_0(X)} B$ and the isomorphism of Lemma \ref{lemma1} is used.
\end{lemma}

\begin{proof}
The map is well-defined since 
if $g H = g' H$ then $g h = g'$ for some $h \in H$ with $g^* g = h h^*$ and
so
$g' (f(g'))= gh(f(gh))= g h \sigma^{-1}_H(h^* h) h^* (f(g))=
g \sigma^{-1}_H(h h^*) h h^* (f(g))= g (f(g h h^*)) =
g (f(g))$
with Lemma \ref{lemma13}.
Injectivity is indicated as $0 = g(f(g))$ implies $0= \sigma^{-1}_H(g^* g) g^*g(f(g))= f(g g^* g)$ for all $g \in G_H$.
Surjectivity is because of
$\mu( \sum_{k \in gH} 1_k \otimes 1_{k^* k} \otimes k^*(a))
= 1_{g H}   \otimes g g^*(b)$ for all $g \in G_H, b \in B$.
The $G$-equivariance is computed by
\begin{eqnarray*}
\mu(k(f)) &=& \sum_{g \in G_H/H} 1_{g H} \otimes g (f(k^*g))\,[k k^* \ge g g^*] \\
&=& \sum_{g \in G_H/H} 1_{k g H} \otimes k g (f(g))\,[k^* k \ge g g^*]
\; = \; k (\mu(f))
\end{eqnarray*}
for all $k \in G$ by Lemma \ref{lemma3}.
\end{proof}

\begin{lemma}  \label{lemma5}
One has an isomorphism
$$\mbox{R}_G^H(A \otimes^{C_0(X)} B) \cong \mbox{R}_G^H(A) \otimes^{C_0(X_H)} \mbox{R}_G^H(B)$$
for all $G$-algebras $A$ and $B$.
\end{lemma}

%

The following lemma is similar to Lemma \ref{lemma8} with a similar 
proof. 

\begin{lemma}    \label{lemma6}
Let $H \subseteq G$ be a finite subinverse semigroup, $A$ a $H$-algebra and $B$ a $G$-algebra.
Then there is a $G$-equivariant isomorphism
$$\tau:Ind_H^G \Big ( A \otimes^{C_0(X_H)} R_G^H (B) \Big ) \rightarrow  \mbox{Ind}_H^G(A) \otimes^{C_0(X)}
B$$
induced by
$$\tau(1_{g} \otimes a \otimes b) =  1_{g} \otimes a \otimes g (b) \quad \in \quad \mbox{Ind}_G^H (A) \otimes^{C_0(X)} \varepsilon(E) \otimes^{C_0(X)} B , $$
where $g \in G_H, a \in A, b \in \mbox{R}_G^H(B) \subseteq \varepsilon(E) \otimes^{C_0(X)} B$
and the isomorphism of Lemma \ref{lemma1} is used.
\end{lemma}

\section{The adjointness relation}

\label{sectionAdjoint}

We recall a known result,   
whose general proof holds also unmodified in the inverse semigroup equivariant
setting.


\begin{lemma}  \label{lemma71}
Every full $G$-Hilbert $B$-module $\calh$ is an imprimitivity bimodule which establishes a $G$-equivariant Morita equivalence
between $\calk(\calh)$ and $B$. Hence the element $[(\calh,0)] \in KK^G(\calk(\calh),B)$ is invertible.
\end{lemma}

%


\begin{proposition}    \label{propAdjoint}

If we allow only fibered $C^*$-algebras then
the functor $\mbox{Ind}_H^G$ is left adjoint to the functor $R_G^H$. In other words,
one has an isomorphism
$$KK^G( \mbox{Ind}_H^G(A), B) \cong KK^H(A, R_G^H(B))$$
which is natural in $A$ and $B$
for all $H$-algebra $A$ and fibered $G$-algebras $B$.
\end{proposition}


\begin{proof}
In this proof we restrict $C^*_G$ and the object class of $KK^G$ to fibered $G$-algebras.


We shall consider two projections of adjunction. One is the transformation
$\iota$ of the functors $\mbox{id}_{C^*_H}$ and $\mbox{R}_G^H \mbox{Ind}_H^G$
given by the family of homomorphisms 
$$\iota_A: A \rightarrow \mbox{R}_G^H \mbox{Ind}_H^G(A), \quad \iota_A(a)= 1_{\varepsilon(E_H)} \otimes
\big ( g \mapsto \sigma^{-1}_H(g^* g) g^*(a) \, [g \in H] \big ) $$
%
for $A$ in $C^*_H$ and $a \in A, g \in G_H$. In other words, we may say that
\begin{eqnarray}
\iota_A(a) &=& \sum_{h \in H}   1_{h h^*} \otimes 1_h \otimes 
\sigma^{-1}_H(h^*h) h^*(a)  \label{ident4}  \\
&=& \sum_{h \in H/H}   1_{h h^*} \otimes \sum_{k \in h H} 1_k \otimes  \sigma^{-1}_H(k^*k) k^*(a)   \label{ident5}.
\end{eqnarray}
Line (\ref{ident4}) is the imprecise notation as the summands are actually not elements of $\mbox{R}_G^H \mbox{Ind}_H^G(A)$,
and (\ref{ident5}) is the correct meaning.

Two is the transformation $\pi$ of the functors $\mbox{Ind}_H^G \mbox{R}_G^H$ and $\mbox{id}_{C^*_G}$ realized by
the 
family of morphisms in $KK^G$,
$$
\xymatrix{
\pi_B:  \mbox{Ind}_H^G R_G^H(B) \ar[r]^{\mu}  & c_0(G_H/H) \otimes^{C_0(X)} B
\ar[rr]^{m \otimes \mbox{id} } &&  \calk(\ell^2(G_H/H)) \otimes^{C_0(X)} B  \ar[d]^{\equiv}   \\
& B && \varepsilon(E) \otimes^{C_0(X)} B ,
\ar[ll]_{\cong}
}
$$
indexed by $B$ in $C^*_G$.
Here, $\mu$ is the map of Lemma \ref{lemma8}, $m: c_0(G_H/H) \rightarrow \calk(\ell^2(G_H/H))$
the canonical 
embedding into the diagonal, which is a $G$-equivariant homomorphism,
and the vertical arrow is 
induced by the Morita equivalence of Lemma \ref{lemma71}.
The last isomorphism is by Definition \ref{deffibered} and Lemma \ref{lemma1}.

It is sufficient to show that
\begin{equation}   \label{ident1}
\pi_{\mbox{Ind}_H^G(A)} \circ \mbox{Ind}_H^G(\iota_A) = \mbox{id}_{\mbox{Ind}_H^G(A)}
\end{equation}
in $KK^G$ and
\begin{equation}   \label{ident2}
\mbox{R}_G^H(\pi_B) \circ \iota_{\mbox{R}_G^H(B)} = \mbox{id}_{\mbox{R}_G^H(B)}
\end{equation}
in $KK^H$ by \cite[IV.1 Theorem 2.(v)]{maclane}.

Now
$\mbox{Ind}_H^G(\iota_A) : \mbox{Ind}_H^G(A) \rightarrow \mbox{Ind}_H^G \mbox{R}_G^H \mbox{Ind}_H^G (A)$
is determined by
$$\mbox{Ind}_H^G(\iota_A)(a) = \sum_{g \in G_H} \sum_{h \in H}  1_g \otimes 1_{h h^*} \otimes  1_h \otimes
\sigma^{-1}_H(h^*h) h^*(a(g)) $$
for $a \in \mbox{Ind}_G^H(A)$. 
Then, for $\mu$ of Lemma \ref{lemma8},
\begin{eqnarray*}
&& \mu \big(\mbox{Ind}_H^G(\iota_A)(a) \big ) \\
 &=& 
\sum_{g \in G_H/H} \sum_{h \in H} 1_{g H} \otimes 1_{g h h^* g^*} \otimes 1_{g h}  \otimes
\sigma^{-1}_H(h^*h) h^*(a(g)) \,[g^*g \ge h h^*] \\   
&=& \sum_{g \in G_H/H} \sum_{h \in H} 1_{g H} \otimes 1_{g h}  \otimes
a(g h) \,[g^*g = h h^*] \\
&=& \sum_{g \in G_H/H} \sum_{k \in g H} 1_{g H} \otimes 1_{k}  \otimes
a(k)  \quad
%
\in   \quad  c_0(G_H/H) \otimes^{C_0(X)} \mbox{Ind}_H^G(A) ,
\end{eqnarray*}
where the second identity uses 
the isomorphism of Lemma \ref{lemma1}, and observe that $\sigma^{-1}_H(h^*h) h^*(a(g)) = a( gh)$ 
by Definition \ref{defInductionAlg}.

Now $\pi_{\mbox{Ind}_H^G(A)} \circ  \mbox{Ind}_H^G(\iota_A) : \mbox{Ind}_H^G (A) \rightarrow
\mbox{Ind}_H^G (A)$ is the Kasparov cycle
\begin{equation}   \label{cycle1}
(\rho, \ell^2(G_H/H) \otimes^{C_0(X)} \mbox{Ind}_H^G(A),0),
\end{equation}
where
$$\rho:\mbox{Ind}_H^G(A) \rightarrow \call \big(\ell^2(G_H/H) \otimes^{C_0(X)} \mbox{Ind}_H^G(A) \big ) $$
is the multiplication operator
$$\rho(a)(\xi \otimes v) = 
\sum_{g \in G_H/H}  \sum_{k \in g H} \xi(g H) 1_{g H} \otimes 1_{k}  \otimes
a(k) \,v(k) 
$$
for $\xi \otimes v \in \ell^2(G_H/H) \otimes^{C_0(X)} \mbox{Ind}_H^G(A)$.

Observing the image of $\rho$,
by a standard argument we may cut down the Hilbert module $\calh$ of the Kasparov cycle
(\ref{cycle1}) to the 
Hilbert submodule
$$\calh_0 = \overline{\mbox{span}} \Big \{ 1_{g H} \otimes \sum_{k \in gH} 1_{k} \otimes a_k  \in \calh |
\, g\in G_H/H, a_k \in A  \Big \}$$
and thus obtain an equivalent cycle $(\rho_0,\calh_0,0)$. Here $\rho_0(a)= \rho(a)|_{\calh_0}$.
We have an isomorphism 
$$u:\calh_0 \rightarrow \mbox{Ind}_H^G (A)  : u \Big (  1_{g H} \otimes \sum_{k \in g H} 1_{k} \otimes a_k \Big ) =
\sum_{k \in g H} 1_{k} \otimes a_k$$ 
of $G$-Hilbert $\mbox{Ind}_H^G (A)$-modules.
This transformation yields another equivalent Kasparov cycle $(i,\mbox{Ind}_H^G(A),0)$, where $i$ is the multiplication
operator on $\mbox{Ind}_H^G(A)$.
Hence 
(\ref{ident1}) is verified.

We are going to show (\ref{ident2}).
We have
$$\iota_{\mbox{R}_G^H(B)}(b)= \sum_{h \in H}   1_{h h^*} \otimes 1_h \otimes 
\sigma^{-1}_H(h^* h) h^*(b) $$
for $b \in {\mbox{R}_G^H(B)}$ and
\begin{eqnarray}
\mbox{R}_G^H(\mu) \big (\iota_{\mbox{R}_G^H(B)}(b) \big )
&=& \sum_{h \in H / H} 1_{h h^*} \otimes 1_{h H} \otimes 
\sigma^{-1}_H(h h^*) (b)   \label{ident3} \\
&\in&  \mbox{R}_G^H \big (c_0(G_H/H) \otimes^{C_0(X)} \mbox{R}_G^H(A)  \big ) .  \nonumber
\end{eqnarray}

Now
$\mbox{R}_G^H(\pi_B) \circ \iota_{\mbox{R}_G^H(B)}: \mbox{R}_G^H(B) \rightarrow \mbox{R}_G^H(B)$
is realized by the Kasparov cycle
\begin{equation}   \label{cycle2}
(\nu, \varepsilon(E_H) \otimes^{C_0(X_H)} \mbox{Res}_G^H \big (\ell^2(G_H/H) \otimes^{C_0(X)} \mbox{R}_G^H(B) \big ),0),
\end{equation}
where $\nu(b)$ is the multiplication operator with the element (\ref{ident3}).
Again, similar as before, we can cut down the Hilbert module of 
this cycle to
\begin{eqnarray*}
\varepsilon(E_H) \otimes^{C_0(X)} \ell^2(H_H/H) \otimes^{C_0(X)} \mbox{R}_G^H(B)
&\cong&
\mbox{R}_G^H(B),
\end{eqnarray*}
where the last isomorphism of $G$-Hilbert $\mbox{R}_G^H(B)$-modules is given by 
$$v(h h^* \otimes 1_{h H} \otimes \sigma^{-1}_H( h h^*) (b) ) = \sigma^{-1}_H(h h^*)(b).$$
Noticing that $\sum_{h \in H/H} \sigma^{-1}_H(h h^*)(b) = b$ by Lemma \ref{lemma13},
we see that the new equivalent Kasparov cycle is
$(j,\mbox{R}_G^H(B),0)$, where $j$ is the multiplication operator.
This shows (\ref{ident2}).
\end{proof}

\begin{example}    \label{example}
{\rm

Let us give a simple example where $G=E$ is finite and consists only of idempotent elements,
and $H=\{e\}$ consists only of a single 
element of $G$.
We obtain by direct computation 
\begin{eqnarray*}
&&  KK^E(\mbox{Ind}_H^E(A),B) = KK^E((\mbox{Ind}_H^E(A))_{\varepsilon_{e}},B)
= KK^E((\mbox{Ind}_H^E(A))_{\varepsilon_{e}},B_{\varepsilon_{e}}) \\
& =&  KK^H(A,R_G^H(B)),
\end{eqnarray*}
verifying Proposition \ref{propAdjoint}.

}
\end{example}


\section{An adaption of a paper of Mingo and Phillips}

\label{sectionAdaption}


\label{sectionMingoPhillips}

In this section we adapt some central results of the paper \cite{mingophillips} by Mingo and Phillips
to the $\ell^2(G)$-space of Definition \ref{definition1}.
%
%
Let $\cale$ be a $G$-Hilbert $B$-module.
Let us write 
$$L^2(G,\cale) = \ell^2(G) \otimes^{C_0(X)} \cale.$$

\begin{lemma}[Cf. Lemma 2.3 of \cite{mingophillips}]    \label{lemmamingophillips23}
If $\cale_1$ and $\cale_2$ are $G$-Hilbert $A$-modules which are isomorphic as Hilbert $A$-modules then $L^2(G,\cale_1)$
and $L^2(G,\cale_2)$ are isomorphic as $G$-Hilbert $A$-modules.
\end{lemma}

\begin{proof}
Let $u \in \call(\cale_1,\cale_2)$ be a unitary operator. Note that $g g^{*} \in \call(\cale_i)$ commutes
with $u$ for all $g \in G$ since $u$ is $A$-linear and $g g^{*}(\xi) a = \xi g g^{*}(a)$ for all $\xi \in \cale_i, a \in A$.
Then it can be checked that $V: L^2(G,\cale_1) \rightarrow  L^2(G,\cale_2)$
given by $V(1_g \otimes \xi) := 1_g \otimes g u g^{*}(\xi)$
defines an isomorphism of $G$-Hilbert $A$-modules. 
We show that $V$ is $G$-equivariant. 
For $h \in G$ we have
\begin{eqnarray*}
h\big(V(1_g \otimes \xi )\big ) &=& 1_{h g} \otimes h g u g^{*} h^{*} h (\xi)  \, [h^* h \ge g g^*] \\
&=& V\big (h(1_g \otimes \xi) \big ),
\end{eqnarray*}
because $h^{*} h \in \call(\cale_i)$ commutes with $g u g^{*} \in \call(\cale_1,\cale_2)$.

For the inner product
note that
\begin{eqnarray*}
\langle V(1_g \otimes \xi),V(1_h \otimes \eta)\rangle
&=& 1_{g g^*} \otimes \langle gug^{*}(\xi), h u h^{*} (\eta) \rangle   \, [h=g] \\
&=& \langle 1_g \otimes \xi,1_h \otimes \eta\rangle
\end{eqnarray*}
with inner product rules.
\end{proof}

\begin{corollary}[Cf. Theorem 2.4 of \cite{mingophillips}]
Let $\cale$ be a $G$-Hilbert $A$-module which is countably generated and full as a Hilbert $A$-module.
Then $(L^2(G,\cale))^\infty$ is isomorphic to $(L^2(G,A))^\infty$ by a $G$-equivariant isomorphism
of Hilbert $A$-modules.
\end{corollary}

\begin{proof}
Same proof as in Mingo and Phillips \cite{mingophillips}, Theorem 2.4, but by applying
Lemma \ref{lemmamingophillips23} instead of \cite[Lemma 2.3]{mingophillips}.
\end{proof}

\begin{corollary}[Cf. Corollary 2.6 of \cite{mingophillips}]    \label{corollarymingophillips26}
Let $A$ be a $G$-algebra and suppose that $A$ has a strictly positive element.
If $p \in \calm(A)$ is a full $G$-invariant projection then
$p \otimes 1 \sim
1 \otimes 1$
(Murray--von Neumann equivalence) in $\calm(A \otimes^{C_0(X)} \calk(L^2(G)^\infty))$
by a $G$-invariant partial isometry.
\end{corollary}

\begin{proof}
The proof of the original goes verbatim through. The usage of the balanced tensor product $\otimes^{C_0(X)}$
instead of $\otimes$ is obligatory.
\end{proof}


\begin{remark}   \label{remark1}
{\rm
In \cite{attempts} we 
considered another model of an $\ell^2(G)$-space, denoted $\widehat{\ell^2}(G)$, over the $G$-algebra $C_0(X)$ which satisfies analogous results
as presented in this section, provided that $G$ is {\em $E$-continuous}.
This essentially means that certain increasing sequences of projections in $C_0(X)$ converge 
pointwise to a projection in $C_0(X)$. We could enlarge any inverse semigroup $G$ to another
bigger $E$-continuous inverse semigroup $\overline G$ containing $G$.
(By adjoining to $G$ the projections corresponding to all subsets of $X$.)
This would yield another $\ell^2(G)$-model, namely $\mbox{Res}_{\overline G}^G(\widehat{\ell^2}(\overline G))$,
now over the $G$-algebra $C_0(X_{\overline G})$.
Even if $\overline G$ is usually uncountable, the module $\widehat{\ell^2}(\overline G)$
is still countably generated (as $\varphi_{g e} = \varphi_g \cdot 1_e$ for $e \in \overline E$).
}
\end{remark}

%
%
%
%
%
%
%

\section{$\varepsilon KK^G$ is triangulated}

\begin{definition}
{\rm
Let $\varepsilon K K^G$ denote the full subcategory of $KK^G$ which is generated by all objects in $KK^G$
which are isomorphic in $KK^G$ 
to a fibered $G$-algebra.
}
\end{definition}

In this section we shall show that $\varepsilon KK^G$ is a triangulated category.
At first we need a Cuntz picture of $\varepsilon KK^G$. To this end we shall adapt Meyer's paper
\cite{meyer} which provides a Cuntz picture of Kasparov theory in the group equivariant case.

A central idea of \cite{meyer} is to make the operator of a Kasparov cycle $G$-invariant
by switching to the $L^2(G)$-version of a Hilbert module, like in \ref{lemmamingophillips23}.
We want to adapt \cite{meyer} to our setting, and that is why we need a model of an $\ell^2(G)$-space
which has nice properties so that the Mingo--Phillips tricks of Section  \ref{sectionMingoPhillips} hold.
It is less difficult to find such an $\ell^2(G)$-space which is a $G$-Hilbert modul over some
$G$-algebra $B$. 
By what we observed in Section \ref{sectionMingoPhillips} we could use $\ell^2(G)$
of Definition \ref{definition1} or $\mbox{Res}_{\overline G}^G \ell^2(\overline G)$ of Remark \ref{remark1}.
But as in \cite{meyer}, one often needs finally to cancel the $\ell^2(G)$-space
by Morita equivalence as follows. One has given a $G$-algebra $\calk(\ell^2(G)) \otimes A$ and wants to get rid of
$\calk(\ell^2(G))$. Hence one uses Morita equivalence $\calk(\ell^2(G)) \equiv \calk(B) \cong B$
and so $KK^G$-equivalently changes $\calk(\ell^2(G)) \otimes A$ to $B \otimes A$. Now $B$ must 
be the neutral element for the tensor product,
like $B=\C$, such that $B \otimes A \cong A$. That is why we have this constraint on the coefficient
algebra $B$ of the $\ell^2(G)$-module.

The $\ell^2(G)$-space of Definition \ref{definition1} over the $G$-algebra $B= \varepsilon(E)$ satisfies this
constraint for the class of fibered $G$-algebras $A$ and the balanced tensor product
by Lemma \ref{lemma1}.


\begin{lemma}	\label{lemmacone}
\begin{itemize}

\item[(a)]
If $T,A,B$ are $G$-algebras and $T$ is equipped with the trivial $G$-action then
there is a canonical isomorphism
$$(T \otimes A) \otimes^{C_0(X)} B \cong T \otimes ( A \otimes^{C_0(X)} B). $$

\item[(b)]
If $A,B,C$ are $G$-algebras and
$f: A \rightarrow B$ is a $G$-equivariant homomorphism then
$$\mbox{cone}(f) \otimes^{C_0(X)} C \cong \mbox{cone}(f \otimes^{C_0(X)} \mbox{id}_C).$$

\end{itemize}
\end{lemma}


\begin{corollary}
The class of fibered $G$-algebras is closed under taking suspension and cones.
\end{corollary}

\begin{proof}
Set $T=\Sigma=C_0(\R)$ (suspension) and $B= C=\varepsilon(E)$ in the last lemma.
\end{proof}


Our aim is to slightly adapt \cite[Section 5]{attempts} 
by making the following simple modifications:
\begin{itemize}

\item
Replace every occurrence of the $G$-Hilbert $C_0(X)$-module $\widehat{\ell^2}(G)$ in \cite[Section 5]{attempts} 
by the $G$-Hilbert $\varepsilon(E)$-module $\ell^2(G)$ of Definition \ref{definition1}.

\item
Substitute every single occurring $C_0(X)$  in \cite[Section 5]{attempts} which appears as a $G$-algebra in its own right
(not in $\otimes^{C_0(X)}$) by the $G$-algebra $\varepsilon(E)$ (for example in $C_0(X)^\infty$ or in $C_0(X) \oplus \calh$).

\end{itemize}

By an analogous replacement $L^2(G) \rightarrow \widehat{\ell^2}(G)$, $\C \rightarrow C_0(X)$ 
and $\otimes \rightarrow \otimes^{C_0(X)}$,
Section 5 of \cite{attempts} was obtained 
from Meyer's paper \cite{meyer}.

In the next theorem we state a version of 
\cite[Theorem 6.5]{meyer}
in a slightly simplified but less technical form, which summarizes its quintessence.

Since we did not go through all details of the paper \cite{meyer}, we should 
view the following theorem as a conjecture!

%
%

\begin{theorem}[Cf. \cite{meyer}]
Let $A$ and $B$ be 
fibered $G$-algebras
and $x \in KK^G(A,B)$. Then there exist
fibered
$G$-algebras
$A'$ and $B'$, isomorphisms $a \in KK^G(A,A')$ and $b \in KK^G(B,B')$,
and a $G$-equivariant $*$-homomorphism $f:A' \rightarrow B'$ such that
$$x = a \circ f \circ b^{-1}.$$


\end{theorem}

%
%

\begin{proposition}[Cf. \cite{meyernest}]
The category $\varepsilon KK^G$ is triangulated by calling a triangle distinguished if it is isomorphic
to a mapping cone triangle, and by defining the translation functor to be $A \mapsto \Sigma^{-1} A$
(desuspension).
\end{proposition}

For the details of the last proposition see \cite[Section 6]{attempts} and \cite{meyernest}, respectively.
Notice also, that actually a certain direct limit $\widetilde{\varepsilon KK^G}$ induced by the 
suspension functor
is triangulated, and $\varepsilon KK^G$ is just its sloppy notation.
Both categories, however, are equivalent.

%

\section{The Baum--Connes map}

\label{sectionBC}

Throughout this section the fibered restriction functors and induction functors
are understood to be restricted to the category $\varepsilon KK^G$. 
That is we view induction and fibered restriction as $\mbox{Ind}_H^G:KK^H \rightarrow \varepsilon KK^G$
and $\mbox{R}_G^H: \varepsilon KK^G \rightarrow KK^H$.

%

To avoid 
cumbersome
notation, we shall make the following convention:

{\em
Throughout this section we shall exclusively work with the triangulated category $\varepsilon KK^G$
but denote it by $KK^G$
for simplicity most of the time!
}


%

\begin{definition}
{\rm
An {\em exact functor} $F:\cals \rightarrow \calt$ between triangulated categories $\cals$ and $\calt$
is a suspension-intertwining ($F \circ S_\cals = S_\calt \circ F$) functor which sends exact sequences
$S B \rightarrow C \rightarrow A \rightarrow B$ canonically to exact sequences
$S F(B) \rightarrow F(C) \rightarrow F(A) \rightarrow F(B)$ (see \cite{krause}).
A functor $F$ between triangulated categories is called {\em triangulated}
if it is exact and additive.
}
\end{definition}


\begin{lemma}  \label{lemma7}
The fibered restriction functors $\mbox{R}_G^H$ 
and induction functors $\mbox{Ind}_H^G$ 
are triangulated functors.
Also for every $G$-algebra $B$, the balanced tensor product functor $A \mapsto A \otimes^{C_0(X)} B$ from the category $KK^G$ into itself
is triangulated.

\end{lemma}

\begin{proof}
Given an exact triangle in $KK^H$, we may switch to its isomorphic mapping cone triangle 
according to definition \cite[6.4]{attempts}. This mapping cone triangle is sent canonically to a mapping cone triangle
(and hence exact triangle) by the fibered restriction and induction functors by
Lemma \ref{lemmacone} and \cite[4.3]{attempts}.
\end{proof}

For the orthogonal subcategory $\cals^\bot$ of a triangulated subcategory $\cals$ see \cite[4.8]{krause}.
The expression $\langle \cals \rangle$ denotes the generated triangulated subcategory of a subcategory
$\cals$ of a triangulated category.

\begin{definition}[Cf. Definition 4.1 of \cite{meyernest}]   \label{defCompactlyinduced}
{\rm
An object $A$ in $KK^G$ is called {\em compactly induced} 
if there exists an object $B$ in $KK^G$ and a finite subinverse semigroup $H \subseteq G$
such that $A$ is isomorphic to $\mbox{Ind}_H^G(B)$ in $KK^G$.
The full subcategory of $KK^G$ induced by the compactly induced objects is denoted by $\calc \cali$.
}
\end{definition}

\begin{definition}[Cf. Definition 4.1 of \cite{meyernest}]
{\rm
A morphism $f \in KK^G(A,B)$ is called a {\em weak equivalence} 
if $\mbox{R}_G^H(f)$ is invertible in $KK^H$ for all finite subinverse semigroups $H \subseteq G$.
}
\end{definition}

%
\begin{definition}[Cf. Definition 4.5 of \cite{meyernest}]   \label{CIsimplicalApproximation}
{\rm
A {\em $\calc \cali$-simplicial approximation} 
of an object $A$ in $KK^G$
is a weak equivalence $f \in KK^G(B,A)$ such that $B$ is an object in $\langle \calc \cali \rangle$.
}
\end{definition}

\begin{definition}[Cf. Definition 4.5 of \cite{meyernest}]  \label{defDiracMorphism}
{\rm
A {\em Dirac morphism} 
is a $\calc \cali$-simplicial approximation of $\varepsilon(E)$.
}
\end{definition}

\begin{definition}[Cf. Definition 4.1 in \cite{meyernest}]
{\rm
Call an object $A$ in $KK^G$ {\em weakly contractible} if $\mbox{R}_G^H(A) = 0$ in $KK^H$ for all
finite subinverse semigroups $H \subseteq G$.
Write $\calc \calc$ for the full subcategory of $KK^G$ of weakly contractible objects.
}
\end{definition}

\begin{lemma}[Cf. Proposition 4.4 of \cite{meyernest}]      \label{lemmaorthogonalcategory}
We have $\calc \calc = \langle \calc \cali \rangle^{\bot}$.
\end{lemma}

\begin{proof}
One proves this verbatim as in \cite[Proposition 4.4]{meyernest}.
One just needs the adjointness relation of Proposition \ref{propAdjoint}.
%
%
\end{proof}

\begin{lemma}[Cf. Lemma 4.2 of \cite{meyernest}]   \label{lemmaCIlocalizing}
The categories
$\langle \calc \cali \rangle$ and $\calc \calc$ are localizing subcategories of $KK^G$.
They
are closed under forming the balanced tensor product $A \mapsto A \otimes^{C_0(X)} B$ for all
$G$-algebras $B$.
\end{lemma}

\begin{proof}
One proves this like Lemma 4.2 of \cite{meyernest}.
The stability under tensor products follows from Lemmas \ref{lemma5} and \ref{lemma6}.
\end{proof}

\begin{theorem}[\cite{meyerTriangulatedII}]    \label{thmapprox}
There exists a Dirac morphism. Even more, every object in $\varepsilon KK^G$ has a $\calc \cali$-simplicial approximation.
\end{theorem}

\begin{proof}
One proceeds verbatim as in the first two paragraphs of the proof of \cite[Theorem 7.3]{meyerTriangulatedII}, which
handles the discrete group case. 
One just replaces the ordinary restriction functor $\mbox{Res}_H^G$ with the fibered
restriction functor $\mbox{R}_H^G$ everywhere.
It is required that this functor commutes with direct sums, which is satisfied.
Also the necessary fact that the functor $\mbox{R}_G^H$ is right adjoint to the functor $\mbox{Ind}_G^H$
is verified in Proposition \ref{propAdjoint}.
The claim follows then by the verbatim analogous argument given in the paragraph after
\cite[Theorem 7.3]{meyerTriangulatedII}.
\end{proof}

Even if Theorem \ref{thmapprox} shows already that every object allows a $\calc \cali$-simplicial approximation,
we shall also demonstrate
how this fact can already be deduced from the existence of a Dirac morphism
by tensoring with a coefficient algebra. This is the next lemma 
and its corollary.

\begin{lemma}[Cf. Theorem 4.7 of \cite{meyernest}]   \label{theoremDiracexactsequence}
Let $D \in  KK^G(P,\varepsilon(E))$ be a Dirac morphism with $P \in \langle \calc \cali \rangle$.
Then there exists an exact triangle
\begin{equation}   \label{diractriangle}
\xymatrix{
P \ar[r]^D & \varepsilon(E) \ar[r]  & N \ar[r] & \Sigma^{-1} P}
\end{equation}
in $KK^G$ with $N \in \calc \calc$.
For every fibered $G$-algebra $A$ this induces canonically by tensoring
an exact triangle
\begin{equation}   \label{exactsequencediracA}
\xymatrix{
P \otimes^{C_0(X)} A \ar[r]^{D \otimes \mbox{id}}  & \varepsilon(E) \otimes^{C_0(X)} A \ar[r]
& N \otimes^{C_0(X)} A \ar[r] & \Sigma^{-1} (P \otimes^{C_0(X)} A)}
\end{equation}
in $KK^G$. 
Here, one has $P \otimes^{C_0(X)} A \in \langle \calc \cali \rangle$
and $N \otimes^{C_0(X)} A \in \calc \calc$.

\end{lemma}


\begin{proof}
By the axioms of a triangulated category, the morphism $D$ from $KK^G$ fits into an exact triangle
as in (\ref{diractriangle}) for some object $N$ in $KK^G$.
Since $\mbox{R}_G^H$ is an exact functor by Lemma \ref{lemma7}, this triangle canonically induces exact triangles in $KK^H$
via $\mbox{R}_G^H$ for all finite subinverse semigroups $H$ in $G$.
By Definition \ref{defDiracMorphism}, $\mbox{R}_G^H(D)$ is an isomorphism in $KK^H$, and so
$\mbox{R}_G^H(N)$ vanishes in $KK^H$ by Corollary 1.2.4
and Remark 1.1.21 of \cite{neemanbook}, or confer \cite[Lemma 2.2]{meyernest}. But this means that $N \in \calc \calc$.

By Lemma \ref{lemma7}
the sequence (\ref{exactsequencediracA}) is exact.
The last claim follows from Lemma \ref{lemmaCIlocalizing}.
%
\end{proof}

\begin{remark}
{\rm
The importance of Lemma \ref{theoremDiracexactsequence} is that its validity is equivalent
to the existence of an exact localization functor $L:KK^G \rightarrow KK^G$
with kernel $\langle \calc \cali \rangle$, see for example Proposition 4.9.1 in \cite{krause}.
This implies the existence of an exact colocalization functor $\Gamma:KK^G \rightarrow KK^G$
with kernel $\calc \calc$ and an equivalence
$KK^G/\calc \calc \cong \langle \calc \cali \rangle$ in the opposite category of $KK^G$, see for example
Propositions 4.12.1, 4.10.1 and 4.11.1 in \cite{krause} together with Lemma \ref{lemmaorthogonalcategory}.
Confer also Proposition 2.9 and the remarks after Definition 4.2 in \cite{meyernest}. 
The complementarity condition of \cite[Definition 2.8]{meyernest} is satisfied
by \cite[Proposition 4.10.1]{krause} and Lemma \ref{lemmaorthogonalcategory},
which combine to $\mbox{Im} L = \langle \calc \cali \rangle^\bot = \calc \calc$.
}
\end{remark}

\begin{corollary}    \label{corollaryExistenceCIsimplicalapprox}
Every object $A$ in $\varepsilon KK^G$ has a $\calc \cali$-simplicial approximation,
for example
$D \otimes \mbox{id}$ 
of (\ref{exactsequencediracA}).

\end{corollary}

\begin{proof}
Let $D \in KK^G(P,\varepsilon(E))$ be a Dirac morphism as in Lemma \ref{theoremDiracexactsequence}.
Since by Definition \ref{defDiracMorphism} $\mbox{R}_G^H(D)$ is an isomorphism in $KK^H$
for every finite subinverse semigroup $H \subseteq G$,
$\mbox{R}_G^H(D \otimes \mbox{id}) = \mbox{R}_G^H(D) \otimes \mbox{R}_G^H(\mbox{id})$
(see Lemma \ref{lemma5})
is also an isomorphism. Hence $D \otimes \mbox{id}$ is a weak equivalence.
\end{proof}

\begin{definition}[Baum--Connes map]    \label{defBaumConnesmap}
{\rm

Let $A$ be an object in $\varepsilon KK^G$. (That is, $A$ is a $G$-algebra which is 
isomorphic in $KK^G$
to a fibered $G$-algebra.)
Choose a $\calc \cali$-simplicial approximation $D \in KK^G(\tilde A,A)$ for $A$. 
(That is, $\tilde A$ is a kind of an approximation of $A$ which in the easiest case is an induced
algebra $\tilde A = \mbox{Ind}_H^G(B)$.)
For a
given descent functor $j^G:KK^G \rightarrow KK$ (there are several choices corresponding
to different crossed products)
form the morphism
$$j^G(D) \in KK(\tilde A \rtimes G, A \rtimes G)$$
as a potentially good approximation between $\tilde A \rtimes G$ and $A \rtimes G$.

The {\em Baum--Connes map} with coefficient algebra $A$ (and with respect to the descent functor $j^G$)
is defined to be the abelian group homomorphism
$\nu^G:K(\tilde A \rtimes G) \rightarrow K( A \rtimes G)$
as indicated in the following commuting diagram:
$$
\xymatrix{
K(\tilde A \rtimes G) \ar[rr]^{\nu^G} \ar[d]^\cong  && K(A \rtimes G)  \\
KK(\C,\tilde A \rtimes G) \ar[rr]^{\otimes j^G(D)}  &&
KK(\C,A \rtimes G) \ar[u]^\cong
}
$$
The vertical arrows are the usual isomorphisms \cite[\S 6.3]{kasparov1981} and the bottom arrow is
the map which takes the Kasparov product
with $j^G(D)$.

}
\end{definition}

Definition \ref{defBaumConnesmap} does not depend on the choice of
the $\calc \cali$-simplicial approximation
$D$,
see Proposition 2.9.2 of \cite{meyernest}.

\begin{remark}   \label{remark3}
{\rm

\begin{itemize}

\item[(a)]


Notice that in the case of Sieben's crossed product, the domain $K(\tilde A \rtimes G)$ of the Baum--Connes map
is potentially computable by homological means in a triangulated category
as explained in the introduction 
by (\ref{comp1}), Theorem \ref{thm1} and
\cite[Theorem 5.1]{meyerTriangulatedII} applied to
the functor
$F(A) = K(A \rtimes G)$ from $KK^G$ to the abelian groups. 
%

%
%
%
%

\item[(b)]

For the full crossed product the constructed Baum-Connes map 
is usually 
less powerful. 
Indeed, for instance if $E$ has a minimal element $e_0$ then $\C$ is a fibered $G$-algebra.
(If $E$ has no minimal element it may be adjoined to $G$ by setting $G':= G \sqcup \{e_0\}$
and $e_0 g = g e_0 = e_0$.)
Note that $e_0 G = e_0 G e_0$ is a subgroup of $G$ and
$KK^G(\C,A) = KK^{e_0 G}(\C,A_{\varepsilon_{e_0}})$.
Hence, if $e_0 G$ is the trivial group then $\mbox{Ind}_{\{e_0\}}^G (\C) = \C$ is a Dirac element and
the Baum--Connes map $K(\tilde \C \rtimes G) \rightarrow K(\C \rtimes G)$ is trivially the identity map.
In any case, at the end of the day we miss a Green imprimitivity theorem as in Theorem \ref{thm1}
for potential computation of the domain 
of the Baum-Connes map.

\end{itemize}

}
\end{remark}


%
%
%
%

\bibliographystyle{plain}
\bibliography{references}

\begin{thebibliography}{10}

\bibitem{baumconneshigson1994}
P.~Baum, A.~Connes, and N.~Higson.
\newblock {Classifying space for proper actions and $K$-theory of group
  $C\sp*$- algebras.}
\newblock {Doran, Robert S. (ed.), $C\sp*$-Algebras: 1943-1993. Providence, RI:
  American Mathematical Society. Contemp. Math. 167, 241-291 (1994).}

\bibitem{burgiGreen}
B.~Burgstaller.
\newblock {An elementary Green imprimitivity theorem for inverse semigroups}.
\newblock {preprint arXiv:1405.1619}.

\bibitem{attempts}
B.~Burgstaller.
\newblock {Attempts to define a Baum--Connes map via localization of categories
  for inverse semigroups}.
\newblock {preprint arXiv:1506.08412}.

\bibitem{burgiUniversalKK}
B.~Burgstaller.
\newblock {The universal property of inverse semigroup equivariant
  $KK$-theory}.
\newblock {preprint arXiv:1405.1613}.

\bibitem{burgiSemimultiKK}
B.~Burgstaller.
\newblock {Equivariant $KK$-theory for semimultiplicative sets}.
\newblock {\em New York J. Math.}, 15:505--531, 2009.

\bibitem{green1978}
P.~{Green}.
\newblock {The local structure of twisted covariance algebras.}
\newblock {\em {Acta Math.}}, 140:191--250, 1978.

\bibitem{kasparov1981}
G.G. Kasparov.
\newblock {The operator K-functor and extensions of C*-algebras.}
\newblock {\em Math. USSR, Izv.}, 16:513--572, 1981.

\bibitem{kasparov1995}
G.G. {Kasparov}.
\newblock {$K$-theory, group $C^*$-algebras, and higher signatures
  (conspectus).}
\newblock In {\em {Novikov conjectures, index theorems and rigidity. Vol. 1.}},
  pages 101--146. Cambridge University Press, 1995.

\bibitem{khoshkamskandalis2004}
M.~Khoshkam and G.~Skandalis.
\newblock {Crossed products of $C^*$-algebras by groupoids and inverse
  semigroups.}
\newblock {\em J. Oper. Theory}, 51(2):255--279, 2004.

\bibitem{krause}
H.~Krause.
\newblock {Localization theory for triangulated categories.}
\newblock {Holm, Thorsten (ed.) et al., Triangulated categories. Based on a
  workshop, Leeds, UK, August 2006. Cambridge: Cambridge University Press.
  London Mathematical Society Lecture Note Series 375, 161-235 (2010).}, 2010.

\bibitem{maclane}
S.~{Mac Lane}.
\newblock {\em {Categories for the working mathematician. 2nd ed.}}
\newblock New York, NY: Springer, 2nd ed edition, 1998.

\bibitem{meyer}
R.~{Meyer}.
\newblock {Equivariant Kasparov theory and generalized homomorphisms.}
\newblock {\em {$K$-Theory}}, 21(3):201--228, 2000.

\bibitem{meyerTriangulatedII}
R.~{Meyer}.
\newblock {Homological algebra in bivariant $K$-theory and other triangulated
  categories. II.}
\newblock {\em {Tbil. Math. J.}}, 1:165--210, 2008.

\bibitem{meyernest}
R.~{Meyer} and R.~{Nest}.
\newblock {The Baum-Connes conjecture via localisation of categories.}
\newblock {\em {Topology}}, 45(2):209--259, 2006.

\bibitem{mingophillips}
J.A. Mingo and W.J. Phillips.
\newblock {Equivariant triviality theorems for Hilbert $C\sp*$-modules.}
\newblock {\em Proc. Am. Math. Soc.}, 91:225--230, 1984.

\bibitem{neemanbook}
A.~Neeman.
\newblock {\em {Triangulated categories.}}
\newblock {Annals of Mathematics Studies. 148. Princeton, NJ: Princeton
  University Press. vii, 449 p.}, 2001.

\bibitem{paterson}
A.L.T. Paterson.
\newblock {\em {Groupoids, inverse semigroups, and their operator algebras.}}
\newblock {Progress in Mathematics (Boston, Mass.). 170. Boston, MA:
  Birkh\"auser.}, 1999.

\bibitem{sieben1997}
N.~Sieben.
\newblock {$C^*$-crossed products by partial actions and actions of inverse
  semigroups.}
\newblock {\em J. Aust. Math. Soc., Ser. A}, 63(1):32--46, 1997.

\end{thebibliography}

\end{document}